\documentclass[11pt, oneside]{amsart}

\newcommand{\hkra}{\hookrightarrow}

\newcommand{\ra}{\rightarrow}

\newcommand{\rimp}{\Rightarrow}

\newcommand{\set}[1]{\left\{ #1 \right\}}

\newcommand{\cc}[1]{\overline{#1}}

\newcommand{\union}[2]{\bigcup\limits_{#1}{#2}}
\newcommand{\sub}{\subseteq}

\newcommand{\sm}{\ensuremath{\setminus}}

\newcommand{\s}[2]{\sum\limits_{#1}{#2} }

\newcommand{\inv}{^{-1}}

\newcommand{\RE}{\mathbb{R}}

\newcommand{\bC}{\mathbb{C}}

\newcommand{\bH}{\mathbb{H}}

\newcommand{\bN}{\mathbb{N}}

\newcommand{\bR}{\mathbb{R}}
\newcommand{\bS}{\mathbb{S}}

\newcommand{\bZ}{\mathbb{Z}}
\newcommand{\cA}{\mathcal{A}}
\newcommand{\cB}{\mathcal{B}}

\newcommand{\cD}{\mathcal{D}}
\newcommand{\cE}{\mathcal{E}}
\newcommand{\cF}{\mathcal{F}}

\newcommand{\cH}{\mathcal{H}}

\newcommand{\cM}{\mathcal{M}}

\newcommand{\cP}{\mathcal{P}}

\newcommand{\cW}{\mathcal{W}}

\newcommand{\cY}{\mathcal{Y}}

\newcommand{\cl}{\colon}

\newcommand{\lbr}[1]{\Bigl(#1\Bigr)}

\newcommand{\emst}{\emptyset}

\newcommand{\xra}{\xrightarrow}

\newcommand{\pd}{\partial}

\newcommand{\eva}{\normalfont\text{ev}}

\newcommand{\chc}{_{*}}

\newcommand{\hpd}{\cc{\partial}}

\newcommand{\dsm}{\oplus}

\newcommand{\ucc}{^{*}}

\newcommand{\Tho}{\normalfont\text{Th}}
\newcommand{\Indv}{\normalfont\text{Ind}}

\newcommand{\smap}{\wedge}

\usepackage[utf8]{inputenc}
\usepackage[english]{babel}
\usepackage{amsmath}
\usepackage{amssymb}
\usepackage{amsfonts}
\usepackage{amsthm}
\usepackage{enumerate}
\usepackage[dvipsnames]{xcolor}
\usepackage{tikz-cd}
\usepackage{subfiles}
\usepackage[scr]{rsfso}
\usepackage{extarrows}
\usepackage{hyperref}
\usepackage{appendix}
\usepackage{bm}
\hypersetup{
	colorlinks=true,
	linkcolor=Blue,
	filecolor=red,      
	urlcolor=red,
	citecolor=teal,
}

\usepackage{tikz}
\usepackage{tikz-cd}
\usetikzlibrary{%
	matrix,%
	calc,%
	arrows%
}

\newtheorem{de}{Definition}[section]
\newtheorem{thm}[de]{Theorem}
\newtheorem{prop}[de]{Proposition}
\newtheorem{lem}[de]{Lemma}
\newtheorem{cor}[de]{Corollary}
\newtheorem{assumption}[de]{Assumption}
\newtheorem{rem}[de]{Remark}

\newtheorem{conjecture}[de]{Conjecture}

\numberwithin{equation}{subsection}

\makeatletter
\@addtoreset{equation}{section}
\makeatother

\newcommand{\Addresses}{{
  \bigskip
  \footnotesize

  \textsc{Institut de Mathématiques de Jussieu, Sorbonne Université, 4 place Jussieu, 75005 Paris, France }\par\nopagebreak
  \textit{E-mail address}: \texttt{hirschi@imj-prg.fr}
  
\smallskip
  
    \textsc{Imperial College London, South Kensington Campus, London SW7 2AZ, UK}\par\nopagebreak
    \textit{E-mail address}: \texttt{n.porcelli@imperial.ac.uk}
}}

\DeclareMathOperator{\im}{im}

\DeclareMathOperator{\rank}{rank}

\DeclareMathOperator{\ide}{id}

\DeclareMathOperator{\pr}{pr}

\DeclareMathOperator{\fixp}{Fix}

\DeclareMathOperator{\reg}{reg}

\DeclareMathOperator{\indo}{index}

\DeclareMathOperator{\Sp}{Sp}

\makeindex

\begin{document}
\title[Lagrangian intersections and the cuplength]{Lagrangian intersections and cuplength in generalised cohomology theories}
\author{Amanda Hirschi, Noah Porcelli}

\begin{abstract}We find lower bounds on the number of intersection points between two relatively exact Hamiltonian isotopic Lagrangians. The bounds are given in terms of the cuplength of the Lagrangian in various multiplicative generalised cohomology theories. The intersection of the Lagrangians need not be transverse, however, we require certain orientation assumptions.
    This gives stronger bounds than previous estimates on the number of intersection points for suitable Lagrangians diffeomorphic to $\Sp(2)$ or $\Sp(3)$.
    Our proof uses Lusternik-Schnirelmann theory, following and extending work by Hofer \cite{Ho88}.\end{abstract}

\maketitle
\tableofcontents

 \section{Introduction}
 
 \subsection{Background}
 
    Let $(X,\omega)$ be a symplectic manifold, and let $L$ and $L'$ be two Lagrangian submanifolds in $X$. A classical problem in symplectic topology is to find lower bounds on the number of intersection points between $L$ and $L'$ when they are related by a Hamiltonian diffeomorphism.
    
    This has been studied by many authors, under various different assumptions. When the Lagrangians are assumed to be transverse, there are lower bounds obtained from Floer homology, such as in \cite{Fl88}, \cite{Oh93}, and \cite{FOOO11}, among many other references.\par
    The classical Arnol'd conjecture concerns a special case of this question, where $(X, \omega)$ is the product symplectic manifold $(Y \times Y, \sigma \oplus -\sigma)$ for  some compact symplectic manifold $(Y, \sigma)$, $L$ is the diagonal and $L'$ is the graph of a Hamiltonian diffeomorphism of $Y$. This case has been the subject of much study, both with and without the additional assumption of transverse intersection. See \cite{FO99,Rud99,P16, AB21,Re22,BX22} and the references therein.\par
 
    We will not assume that $L$ and $L'$ are transverse, but instead make the following assumption. It excludes disc and sphere bubbling, guaranteeing the compactness of certain moduli spaces of pseudoholomorphic curves.
    
    \begin{assumption}\label{relativeexactness}Throughout this paper, we will assume that
        \begin{enumerate}
            \item $X$ is either closed or a Liouville manifold.
            \item $L$ is connected, closed, and \emph{relatively exact}, i.e.,
                $$\omega \cdot \pi_2(X, L) = 0.$$
        \end{enumerate}
    \end{assumption}
    Under these assumptions, Floer proved 
    \begin{thm}[\cite{Fl88}]
        If $L$ and $L'$ intersect transversely, there is a lower bound
        $$\#L \cap L' \geq \sum\limits_i \mathrm{Rank}(H_i(L; \bZ/2)).$$
    \end{thm}
    Without the transversality assumption, a version of the Arnol'd conjecture states that
    \begin{conjecture}\label{c}
        Given Assumption \ref{relativeexactness}, there is a lower bound
        $$ \#L \cap L' \geq \mathrm{Crit}(L) $$
        where $\mathrm{Crit}(L)$ is the minimal number of critical points of any smooth map $L \rightarrow \bR$.
    \end{conjecture}
    A standard application of the Weinstein neighbourhood theorem implies that if true, this bound must be sharp.\\

   Lusternik-Schnirelmann theory is a powerful tool for studying (numbers of) critical points or intersection points without any transversality assumptions, in contrast to Morse theory. It has been used in many fields other than symplectic geometry. For example, Klingenberg proved in \cite[Theorem 5.1.1]{Kl78} that any metric on $S^2$ admits at least three closed geodesics. 
   Lusternik-Schnirelmann theory has also been used in contact topology, e.g., by Ginzburg and G\"urel in \cite{GiGu20} to find lower bounds for numbers of Reeb orbits. For a more comprehensive discussion and further applications we refer to \cite{CLOT03} or Chapter 11 in \cite{MS17}. \par
   In particular, this technique can be and has been applied to study Conjecture \ref{c}.
    
    \begin{thm}[{\cite[Theorem 3]{Ho88}, \cite[Theorem 1]{Fl89}, \cite[Theorem 1]{Ho85}}]\label{hof}
        Under Assumption \ref{relativeexactness}, there is a lower bound
        \begin{equation}\label{int-mod-2}\#L \cap L' \geq c_{\bZ/2}(L).\end{equation}
        If $X = T^* L$ is a cotangent bundle with $L$ the zero section, there is a lower bound
        \begin{equation}\label{int-mod-0}\#L \cap L' \geq c_{\bZ}(L).\end{equation}
    \end{thm}
    
    Here $c_{R}(L)$ is the cuplength of $L$ in cohomology with coefficients in $R$ for a ring $R$, defined below. It is a lower bound for the Lusternik-Schnirelmann category, a much less computable invariant given by the minimal cardinality of an open cover by nullhomotopic subsets. In particular, just as the rank of $H^*(X;\bZ)$ provides a lower bound for the number of critical points of a Morse function, the cuplength is a lower bound for the number of critical points of an arbitrary smooth function, \cite[Theorem 21]{Op14}. See \cite{AH16} and \cite{FL23} for further estimates involving the cuplength.
    
    The inequality \eqref{int-mod-2} was shown using Lusternik-Schnirelmann theory by Hofer. Independently, Floer's proof proceeds via Conley indices. In general, \eqref{int-mod-2} and \eqref{int-mod-0} are weaker bounds than the one given by Conjecture \ref{c}. There are examples in which it is a strictly weaker bound, such as \cite[Example 3.7]{Rud99}.\par   
    Similar results have been obtained in the monotone, rather than the relatively exact setting in which we work, by L\^{e}-Ono \cite{VaOn96} and Gong \cite{Go21b}.

\subsection{Main results}
    We extend Theorem \ref{hof}, as well as other results in \cite{Ho88}, to generalised cohomology theories. The strategy is to prove the injectivity of a certain pullback map relating the generalised cohomology of $L$ with the generalised cohomology of a certain moduli space of pseudoholomorphic curves.\par
    
    Fix a ring spectrum $R$, representing a multiplicative generalised cohomology theory $R^*$. The main invariant we will use is the cuplength.
    
    \begin{de} Let $Y$ be a path-connected topological space. The \emph{$R$-cuplength} of $Y$ is the natural number (or $\infty$)
    	\begin{equation*} c_{R}(Y) := \inf\{k \in \bN: \forall \alpha_1,\dots,\alpha_k \in \tilde{R}^*(Y) : \alpha_1 \cdot \dots\cdot \alpha_k = 0\},\end{equation*}
     where $\tilde{R}^*$ denotes the reduced cohomology theory. 
    \end{de}

We use Hofer's convention in \cite{Ho88} regarding the cuplength. It differs by one to other definitions in the literature, such as in \cite{IM04}. It is also called the index of nilpotency, \cite{J78}.
    
    The proof of Theorem \ref{hof} relies on various moduli spaces of pseudoholomorphic curves, which might not be cut out transversely.
    Thus, they do not admit an a tangent bundle; instead one has to work with a virtual vector bundle, induced by the linearisation of the Cauchy-Riemann operator. 
    We make the following assumption throughout the paper, which will allow us to apply a version of Atiyah duality later on.
    
    \begin{assumption}\label{Assumption}
    	This virtual vector bundle is $R$-orientable.
    \end{assumption}
    
    From \cite{Por22} we have the following criteria for Assumption \ref{Assumption} to be satisfied for some generalised cohomology theories.
    
    \begin{prop}[{\cite[Proposition 1.13]{Por22}}]\label{Orientability Holds Often}
    	Assumption \ref{Assumption} holds when
    	\begin{enumerate}
    	    \item $R$ is the Eilenberg-MacLane spectrum $H\bZ/2$.
    	    \item $R$ is the Eilenberg-MacLane spectrum $H\bZ$, and $L$ is (relatively) spin.
    	    \item $R^*$ is complex $K$-theory, and $L$ is spin.
            \item\label{sphere} $R^*$ is real $K$-theory, and $TL$ admits a stable trivialisation over a 3-skeleton of $L$ which extends (after complexification) to a stable trivialisation of $TX$ over a 2-skeleton of $X$ (as a complex vector bundle).
    	\end{enumerate}
    \end{prop}
  
    Fix an $\omega$-compatible almost complex structure $J$. If $X$ is Liouville, we assume $J$ to be convex at infinity. Let $Z$ be the infinite strip $\bR + i[0, 1]$ in the complex plane. By pseudoholomorphic, we will always mean with respect to $J$. We are interested in the following standard moduli space of pseudoholomorphic strips with Lagrangian boundary conditions
    $$\cM_{L, L'} := \set{u \in C^\infty(\bR + i[0,1],X) : \hpd_J u = 0,\;E(u)<  \infty,\; u(\bR)\sub L,\; u(\bR+i)\sub L'}.$$
    Evaluation at 0 defines a map $\pi: \cM_{L, L'} \rightarrow L$. Following Hofer's strategy, we can approximate $\cM_{L, L'}$ with spaces of pseudoholomorphic discs and use Theorem \ref{Main Theorem 1} to prove the following.
    
    \begin{prop}\label{Main Theorem 2} The map $\pi^*\cl R^*(L)\ra R^*(\cM_{L,L'})$ is injective.
    \end{prop}
     
    The main work of the paper, done in \textsection \ref{Section 2}, lies in the proof of the auxiliary result Theorem \ref{Main Theorem 1}. In \textsection \ref{Section 3} we approximate $\cM_{L,L'}$ by moduli spaces of pseudoholomorphic discs. The continuity axiom, which is satisfied by generalised cohomology theories as defined in \textsection \ref{2.1}, allows us to deduce Proposition \ref{Main Theorem 2}. 
     
    As in \cite{Ho88}, we combine this result in \textsection \ref{Section 4} with standard Lusternik-Schnirelmann theory to obtain the following lower bound.
     
    \begin{thm}\label{Intersection Points and Cuplength} Suppose $L$ satisfies Assumption \ref{relativeexactness} and Assumption \ref{Assumption} for a ring spectrum $R$. Then the number of intersection points between $L$ and $L'$ satisfies
    	$$\#L \cap L' \geq c_{R}(L).$$
    \end{thm}
    Note that we do not assume that the Lagrangians $L$ and $L'$ intersect transversally.
    
    \smallskip
 
 \begin{rem} If $L = L_1 \sqcup \dots \sqcup L_k$ is a disjoint union of Lagrangians satisfying Assumption \ref{relativeexactness}, we may write $L' = L'_1 \sqcup \dots \sqcup L'_k$ with $L'_i$ Hamiltonian isotopic to $L_i$. Theorem \ref{Intersection Points and Cuplength} applied to each pair $L_i$ and $L'_i$ shows that 
 \begin{equation*}\#L \cap L' \geq \sum_{i=1}^kc_{R}(L_i).\end{equation*}
 \end{rem}

In \textsection \ref{Examples} we give two examples where Theorem \ref{Intersection Points and Cuplength} represents a stronger bound than what was previously known. This uses computations in \cite{IM04} of the cuplengths of the compact symplectic groups $\Sp(2)$ and $\Sp(3)$ with respect to certain generalised cohomology theories. \\

This shows that refining standard techniques via stable homotopy theory can result in stronger estimates.\\

    As we do not assume the transversality of $L$ and $L'$, to our knowledge there is no analogue of our strategy of proof using the setup in \cite{CoJoSe95,CJS09}. However, it may be possible to use their setup to prove Theorem \ref{Intersection Points and Cuplength} using the strategy of \cite{Go21a} instead.\\
    
    Suppose $X$ is compact and symplectically aspherical. If $\psi$ is a (possibly degenerate) Hamiltonian diffeomorphism of $X$, we can apply Theorem \ref{Intersection Points and Cuplength} to the graph of $\psi$ in $X\times X$ to deduce the Hamiltonian version of this inequality.
    
    \begin{cor}\label{corolary} 
        The number of fixed points of $\psi$ satisfies
        $$\#\fixp(\psi) \geq c_{R}(X).$$
    \end{cor}
   In this setting, Conjecture \ref{c} (which implies Corollary \ref{corolary}) has already been verified; see \cite[Theorem A]{Rud99}, \cite[Corollary 4.2]{OR99} and \cite[Theorem 8.28]{CLOT03}. Lusternik-Schnirelmann theory has even been used to extend such estimates to Hamiltonian homeomorphism as in \cite{BHS21}. In a slightly different direction, Schwarz used the cuplength of the quantum cohomology of $X$ to achieve lower bounds, \cite{S98}. It would be interesting to see if the the cuplength of quantum K-theory, defined in \cite{AMS23}, could be used in a similar way to bound the number of fixed points from below.

\subsection*{Acknowledgements}
    Both authors are grateful to Ailsa Keating and Ivan Smith for valuable discussions and for suggesting this project in the case of the latter. Oscar Randal-Williams pointed out the example of $\Sp(2)$ and gave helpful feedback, as did Jonny Evans. The first author thanks Matija Sreckovic for comments on an earlier draft. The second author would also like to thank Jack Smith, Nick Nho and Sam Frengley for helpful discussions. The authors also thank the anonymous referee for helpful feedback.\\ Both authors were funded by the EPSRC during their graduate studies. The first author is supported by ERC grant No. 864919. The second author is supported by the Engineering and Physical Sciences Research Council [EP/W015889/1]. This work was finished while both authors were resident at the Simons Laufer Mathematical Sciences Institute, supported by the NSF grant No. DMS-1928930.
\section{Homotopy and Fredholm theory}\label{Section 2}

\subsection{Generalised cohomology and Thom spectra}\label{2.1}

For our purposes, it suffices to work with classical spectra as defined in \cite{Rud98} or \cite{Ad95}. However, our definitions require some care, as we will take the generalised (co)homology of spaces which are not necessarily homotopy equivalent to a CW complex. A generalised cohomology theory defined on CW complexes can always be extended to all spaces, but this extension may not be unique. We need an extension that satisfies a certain continuity property, namely the statement of \cite[Lemma 5.2]{AMS21}.\par
Unless otherwise specified, we work with spectra whose level spaces are homotopy equivalent to CW complexes. We denote by $\bS$ the sphere spectrum. All our spaces are assumed to be compactly generated and Hausdorff. Our ring spectra need not be commutative.\\

Given an $\Omega$-spectrum $R$, we define, for a pointed space $X$, the \emph{$n^{\mathrm{th}}$ $R$-homology and $R$-cohomology groups} to be 
$$R_n(X) := \pi_n(X \smap R_n) \quad \quad \quad \quad R^n(X) := [X,R_n]_*$$
respectively, for $n$ in $\bZ$, where $[\cdot, \cdot]_*$ denotes the set of pointed homotopy classes of maps. As $R_n\simeq \Omega^2 R_{n + 2}$, the sets $R_n(X)$ and $R^n(X)$ carry a natural abelian group structure for all $n$.\\ 

We recall the definition of relative $R$-(co)homology.\\
 Given an inclusion $j\cl A\hkra X$ of pointed spaces, we denote by 
	$$C_X A := C_A\cup_j X$$ 
	the \emph{(reduced) mapping cone} of $j$, where $C_A$ is the cone over $A$. We view this as a pointed space, with basepoint the vertex of $C_A$ (or equivalently the basepoint of $X$).

The inclusion $X\hkra C_XA$ is a cofibration and collapsing $X$ induces a natural map $C_XA\ra \Sigma A$. By \cite[Theorem 4.6.4]{tD08} these maps fit into an $h$-coexact sequence
$$A\ra X\ra C_XA \ra \Sigma A \ra \Sigma X \ra \dots$$
In particular, there exists for any pointed space $W$ a long exact sequence of pointed sets
\begin{equation}\label{dold-puppe 1}\dots\ra [\Sigma X,W]_* \ra [\Sigma A,W]_*\ra  [C_XA,W]_*\ra [X,W]_* \ra [A,W]_* \end{equation}

The \emph{relative $R$-(co)homology} of $(X,A)$ is defined by 
$$R\chc(X,A) := R\chc(C_{X}A) \qquad \qquad R^*(X,A) := R^*(C_{X}A).$$
We recover the usual long exact sequence of a pair in $R$-cohomology due to \eqref{dold-puppe 1}. Similarly, there is a long exact sequence of a pair in $R$-homology.\par
In the case of a pair of unpointed spaces $(X, A)$, one simply considers the pair $(X_+,A_+)$, where $\cdot_+$ denotes the addition of a disjoint basepoint.\\

A \emph{ring spectrum} consists of an $\Omega$-spectrum $R$ endowed with both a multiplication map $\mu \cl R\smap R\to R$ and a unit $\iota \cl \bS\to R$ such that the usual associativity and unit diagrams commute up to homotopy. Note that we do not require that this product is homotopy commutative. In this case we can define a cup product 
$$\cdot\; \cl R^*(X)\otimes R^*(X)\ra R^*(X)$$
by letting $\alpha \cdot \beta$ be the composite
\begin{equation}\label{cup product de}X\xra{\Delta} X\smap X \xra{\alpha\smap \beta} R_n\smap R_k \ra  (R\smap R)_{n+k} \xra{\mu}R_{n+k}\end{equation}
for $\alpha\in R^n(X)$ and $\beta\in R^k(X)$, where the map $R_n\smap R_k \ra (R\smap R)_{n+k}$ comes from the construction of the smash product. As the cup product is functorial, it induces a cup prodct 
$$\cdot \cl \tilde{R}^*(X)\otimes \tilde{R}^*(X)\to \tilde{R}^*(X)$$
on reduced $R$-cohomology. Recall that the \emph{cuplength} of a path-connected space $X$ is given by 
c\begin{equation}\label{cuplength de} c_{R}(Y) := \inf\{k \in \bN: \forall \alpha_1,\dots,\alpha_k \in \tilde{R}^*(X) : \alpha_1 \cdot \dots\cdot \alpha_k = 0\}.\end{equation}

An important property of the cuplength is that it provides a lower bound for the Lusternik-Schnirelmann category, a variant of which is defined in \textsection\ref{Section 4}. This is clearly true for singular cohomology; however, when one works outside the setting of CW complexes, the cup product of generalised cohomology theories does not necessarily descend to a map
$$R\ucc(X,A)\otimes R\ucc(X,B) \ra R\ucc(X,A\cup B).$$
Nonetheless, we can make the following two observations, which are sufficient to bound the Lusternik-Schnirelmann category from below and will be used in \textsection\ref{Section 4}.

\begin{rem}\label{Cup Product and Relative Cohomology}
    Suppose $\alpha \in R^n(X)$ and $\beta \in R^m(X)$ admit representatives $\tilde{\alpha}: X \rightarrow R_n$ and
    $\tilde{\beta}: X \rightarrow R_m$
    which send subspaces $A, B \subseteq X$ to the basepoint $*$ respectively. Then, by construction, $\alpha \cdot \beta$ admits a representative $X \rightarrow R_{n+m}$ sending $A \cup B$ to $*$. Induction shows the same result for classes $\alpha_1, \ldots, \alpha_k\in R^*(X)$ that admit representatives $\tilde{\alpha_i}$ mapping subspaces $A_i \subseteq X$ to $*$ respectively. In particular, 
    $$\alpha_1\cdot \ldots\cdot \alpha_k = 0\quad \text{in}\quad R^*(A_1\cup \ldots \cup A_k).$$
\end{rem}

\begin{lem}\label{rel cup}
    Let $(X, d)$ be a compact metric space, with an open cover $U_1, \ldots, U_k$. Let $\alpha_i \in R^{n_i}(X)$ be cohomology classes such that $\alpha_i|_{{U_i}} = 0$ in $R^{n_i}({U}_i)$ for all $i$. Then the product $\alpha_1 \cdot \ldots \cdot \alpha_k$ vanishes in $R^*(X)$.
\end{lem}

\begin{proof} Let $V_1,\dots,V_k$ be an open cover of $X$ so that $\cc{V_i}\sub U_i$ for each $i$. Set $d_i := d(\cdot,\cc{V_i})$. As $A_i:= X\sm U_i$ is disjoint from $\cc{V_i}$ and compact, there exists $\epsilon_i > 0$ so that $d_i\inv([0,\varepsilon_i])\sub U_i$. \par
    Pick maps $\tilde{\alpha}_i: X \rightarrow R_{n_i}$ representing each $\alpha_i \in R^{n_i}(X)$. Since $\alpha_i|_{U_i} = 0$, we can choose nullhomotopies
    $$H_i: U_i \times [0, 1] \rightarrow R_{n_i}$$ 
    such that $H_i(\cdot, 0) \equiv *$ and $H_i(\cdot, 1) = \tilde{\alpha}_i|_{U_i}$.\par 
Define maps $\beta_i: X \rightarrow R_{n_i}$ by
    $$\beta_i(x) := \begin{cases}
        \tilde{\alpha}_i(x) & \mathrm{ if }\; \varepsilon_i \leq d_i(x)\\
        H_i(x, \varepsilon_i\inv d_i(x)) & \mathrm{ if }\; 0 \leq d_i(x) \leq \varepsilon_i.
    \end{cases}$$
    Then $\beta_i \simeq \tilde{\alpha}_i$, via the homotopy $G_i: X \times [0, 1] \rightarrow R_{n_i}$ given by
    $$G_i(x, s) := \begin{cases}
        \tilde{\alpha}_i(x) & \mathrm{ if }\; \varepsilon_i \leq d_i(x)\\
        H_i(x, s + (1-s) \varepsilon_i\inv d_i(x)) & \mathrm{ if }\; 0\leq d_i(x) \leq \varepsilon_i.
    \end{cases}$$

    Hence $\beta_i$ is a representative of $\alpha_i \in R^{n_i}(X)$. $\beta_i$ sends $V_i$ to the basepoint $*$ so $\alpha_i \cdot \ldots \cdot \alpha_k = 0$ in $R^*(X)$ by Remark \ref{Cup Product and Relative Cohomology}.
\end{proof}

\medskip

Suppose now that $X$ is a compact space and that $\xi \cl F\ra X$ is a vector bundle of rank $k$. Let $F_\infty$ be its fibrewise one-point compactification. This is a sphere bundle over $X$ with a canonical section $s_\infty$ given by the point at infinity in every fibre. The \emph{Thom space} of $\xi $ is defined to be the pointed space
$$\Tho(\xi) := C_{F_\infty} \im(s_\infty)$$
and its \emph{Thom spectrum}, written $X^\xi$ or $X^F$, to be the spectrum $\Sigma^{\infty}\Tho(\xi)$. In particular, if $\xi$ is the trivial bundle of rank $k$, then $\Tho(\xi) = \Sigma^kX_+$ and $X^\xi = \Sigma^{\infty + k} X_+$. Note that if $X$ is not a CW complex, $\Tho(\xi)$ might not be either. These are the only spectra whose level spaces are not CW complexes that we will encounter.\par 

For a virtual vector bundle of the form $\eta - \bR^N$, we let $\rank(\eta) - N$ be its \emph{rank} and define its \emph{Thom spectrum} to be
$$X^{\eta - \bR^N} := \Sigma^{-N} X^\eta.$$ 
All our virtual bundles will be of this form, so this definition is sufficient for our purposes.\footnote{The homotopy type of the Thom spectrum only depends on the stable isomorphism class of the virtual vector bundle. Due to compactness, any virtual bundle on $X$ is stably isomorphic to one of the considered form, so it suffices for our applications.}\\

Let $R$ be a ring spectrum and $\xi$ a virtual vector bundle over $X$ of rank $k$. An \emph{$R$-orientation} of $\xi$ is an element $u \in R^k(X^\xi)$ such that for any map $j: \Sigma^k \bS \rightarrow X^\xi$ which is a (stable) homotopy equivalence to a fibre, we have  
$$j^*u = \pm [\iota] \in R^k(\Sigma^k\bS) \cong R^0(\bS).$$ 
where $[\iota]$ is the homotopy class of the unit map. Any trivial bundle is $R$-orientable, and if two out of $\xi$, $\eta$ and $\xi \dsm \eta$ are $R$-oriented, then so is the third.\par 

By the Thom isomorphism theorem, \cite[Theorem V.1.3]{Rud98} any $R$-orientation on a virtual vector bundle $\xi$ over a compact CW complex $X$ induces a natural isomorphism
$$R^{* + k}(X^{\xi})\cong R^*(X).$$
By \cite[Lemma 5.2]{AMS21}, this also holds when $X$ is a compact subset of a manifold $M$, and both $\xi$ and its $R$-orientation are pulled back from $M$.\par

We will need the following form of Atiyah duality, which can be viewed as a form of Poincar\'e duality for generalised cohomology theories.

\begin{thm}[{\cite[Theorem 5.2]{AMS21}}]\label{} Let $M$ be a smooth (not necessarily compact) manifold, possibly with boundary, and suppose $Z \sub M$ is any compact subset. Then there is a canonical isomorphism
	$$R_{-*}(M,M\sm Z) \;\cong\; R^*\lbr{Z^{-TM|_Z}}$$
	compatible with restriction to smaller closed subsets $Z' \subseteq Z$.\par 
\end{thm}

If $M$ is $R$-oriented on a neighbourhood of $Z$ and of dimension $n$, the Thom isomorphism theorem then gives an isomorphism
$$R^{* + n}(Z) \; \cong \; R_{-*}(M, M \sm Z).$$
for compact subsets $Z \sub M$. In this case, we define the \emph{fundamental class of $M$ along $Z$}, denoted $[M]_Z \in R_n(M, M \sm Z)$, to be the image of the unit in $R^0(Z)$ under this isomorphism. The class $[M]_Z$ depends only on the choice of orientation of which there can be more than two.

We will need the following version of the fact that, given an (oriented) compact manifold with boundary $M$, the fundamental class of $\partial M$ has vanishing image in the homology of $M$.

 \begin{lem}\label{Fundamental Class of Boundary}Let $W^{n+1}$ be an $R$-oriented smooth manifold with boundary and $K\sub W$ a compact subset. Then the image of $[\pd W]_{K\cap \pd W}$ in $R_n(W,W\sm K)$ is $0$.\end{lem}
  
\begin{proof} Suppose first that $W$ is compact. The map $\pd$ in the exact sequence of a pair
    $$R_{n+1}(W, \pd W) \xrightarrow{\pd} R_n(\pd W) \rightarrow R_n(W)$$
    sends $[W]$ to $[\pd W]$ - see \cite[Remark V.2.14.a)]{Rud98}. The claim with $K = W$ then follows from the exactness of this sequence.\par
     
    Now assume $W$ is non-compact and set $C = K \cap \pd W$.
    By excision, we may modify $W$ away from $K$, and replace $W$ with a compact smooth neighbourhood of $K$.
    Then $[\pd W]_C$ is the restriction of the fundamental class $[\pd W] \in R_n(\pd W)$. We may deduce the claim now from the first step and the commutativity of the following diagram.
    $$\begin{tikzcd}
        R_{n+1}(W, \partial W) \arrow[r, "\partial"] & R_n(\partial W) \arrow[r] \arrow[d] & R_n(W) \arrow[d] \\
        & R_n(\partial W, \partial W \setminus C) \arrow[r] & R_n(W, W \setminus K)
    \end{tikzcd}$$
\end{proof}

\subsection{Index bundles}

Let $s \cl \cB\ra \cE$ be a smooth Fredholm section of a Banach bundle over a Banach manifold intersecting the zero section of $\cE$ transversely. By the infinite-dimensional implicit function theorem \cite[Theorem A.3.3]{MS04}, the zero locus $s\inv(0)$ is a smooth manifold of dimension $\indo(Ds) = \dim(\ker(Ds))$ with tangent bundle $\ker(Ds) \ra s\inv(0)$. If $Ds$ is not fibrewise surjective, the zero locus is not necessarily smooth or may have excess dimension. The natural replacement of $\ker(Ds)$ in this case is the \emph{index bundle}, a virtual vector bundle constructed below. It relies on the notion of the stabilisation of a Fredholm operator.

\begin{de} Let $D\cl X\ra Y$ be a Fredholm operator between two Banach spaces. We call an operator $T\cl \bR^N \ra Y$ a \emph{stabilisation of $D$} if $D+ T\cl X\dsm \bR^N \ra Y$ is surjective.
\end{de}

As $T$ is compact, $D+T$ is still a Fredholm operator and 
$$\indo(D+T) = \indo(D)+N$$ 
by \cite[Theorem A.1.5(i)]{MS04}. Given a smoothly varying family of Fredholm operators we will show the existence of a smoothly varying family of stabilisations near a compact subset in Lemma \ref{Stabilisation over Compact Subset}.\par 

Let us fix our setting for the rest of this subsection. Let $Y$ be a separable Hilbert manifold, $\cH$ a separable Hilbert space and $\Lambda$ a compact finite-dimensional manifold with boundary.  We assume that $\psi \cl V \ra \cH$ is a smooth Fredholm map with $V\sub Y\times\Lambda$ an open subset. Define the open subset $$V^{\reg} := \{(x,\lambda)\in V: d_1\psi(x,\lambda) \text{ is surjective}\}$$
where $d_1$ is the derivative with respect to the first argument.

\begin{lem}\label{Stabilisation over Compact Subset} For any closed subset $A\sub V^{\reg}$ and any neighbourhood $U \sub V$ of a compact subset $K\sub V$, there exists a smooth map $T\cl V\times \RE^k\ra \cH$ such that 
	\begin{enumerate}
		\item $T(z,\cdot)$ is linear for each $z\in V$;
		\item $T(z,\cdot) = 0$ for $z \in A$;
		\item $T(z,\cdot) = 0$ for $z \in V\sm U$;
		\item $d_1\psi(x,\lambda) + T(x,\lambda,\cdot)\cl T_{(x, \lambda)}V \dsm\RE^k \ra \cH$ is surjective for $(x,\lambda) \in U$.
	\end{enumerate}
\end{lem}

\begin{proof} For each $z \in K$ there exists an open neighbourhood $U_z \sub V$ of $z$, an integer $k_z \geq 0$, and an operator $T_z \cl \RE^{k_z}\ra \cH$ such that 
	$$d_1\psi(y,\mu) + T_{z}  \cl T_yY \dsm\RE^{k_z} \ra \cH$$
	is surjective for $(y,\mu) \in U_z$. Let $Z \sub K$ be a finite subset such that $U:= \union{z\in Z}{U_z}$ contains $K$ and set $k := \s{z\in Z}{k_z}$. Using a smooth partition of unity subordinate to $\{U_z\}_{z\in Z}\cup \{V\sm K\}$, we obtain an operator $T' \cl V\times \bR^k\ra\cH$ satisfying all conditions save for the second one. Multiplying $T'$ with a smooth bump function which is identically one on $V\sm V^{\reg}$ and vanishes on $A$, we obtain the desired map.
\end{proof}

\begin{de} A family of operators $T$ as in Lemma \ref{Stabilisation over Compact Subset} is said to be a \emph{stabilisation of $\psi$ along $K$ relative to $A$, of rank $k$}. We call $$\Indv_K(\psi;T) := \ker(d_1\psi + T)|_U - \bR_U^k$$ 
	the \emph{index bundle of $\psi$ along $K$ (with respect to $T$)}, defined over a neighbourhood $U$ of $K$.
\end{de}

\begin{lem}\label{Dependent Index Bundles Equivalent over Intersection} Any two index bundles of $\psi$ along $K$ are stably equivalent as germs near $K$. 
\end{lem}

\begin{proof} Suppose $T$ and $S$ are two stabilisations along $K$. We may assume without loss of generality that $d_1\psi +T$ and $d_1\psi +S$ are surjective over the same subset. As we may add factors of $\bR$ to the domain of $T$, respectively $S$, without changing the index bundle, we may assume that $T$ and $S$ are smooth maps $V\times \RE^{k+\ell}\ra \cH$ with $T$ vanishing on $V\times \RE^k\times\{0\}$ and $S$ vanishing on $V\times\{0\} \times\RE^{\ell}$. Now we can linearly interpolate between them and apply \cite[Theorem 14.3.2]{tD08}.
\end{proof}

\begin{de} We let $\Indv_K(\psi)$ denote the stable equivalence class of any $\Indv_K(\psi;T)$ and call it the \emph{index bundle of $\psi$ along $K$}.
\end{de}

\begin{de}\label{Orientability of Fredholm Map} Given a ring spectrum  $R$, the map $\psi$ is \emph{$R$-orientable along $K$} if $\Indv_K(\psi)$ is $R$-orientable on a neighbourhood of $K$. We say that $\psi$ is \emph{$R$-orientable} if it is $R$-orientable along any compact subset. 
\end{de}

\subsection{Proof of Theorem \ref{Main Theorem 1}}

The following result generalises \cite[Theorem 5]{Ho88} to arbitrary ring spectra. 

\begin{prop}\label{Injectivity of Evaluation Map with Spectra} Let $\cY$ be a smooth separable Hilbert manifold and $\cH$ be a separable Hilbert space. Let $\psi: \cY \times [0, 1] \rightarrow \cH$ be a smooth Fredholm map of index $n + 1$, and write $\psi_t$ for its restriction to $\cY \times \{t\}$. Given a ring spectrum $R$, assume
	\begin{enumerate}[\normalfont 1.]
		\item\label{proper} $\psi$ is proper with respect to a neighbourhood of $0$ in $\cH$ and $R$-orientable along $\psi\inv(0)$,
		\item\label{submersive} $\psi_0$ is submersive near $\psi_0\inv(0)$,
		\item\label{other-manifold} there exists a smooth map $\pi \cl \cY\ra N$ to a connected closed manifold $N$ such that 
		$$\pi|_{\psi_0\inv(0)}\cl \psi_0\inv(0)\ra N$$ is a diffeomorphism.
	\end{enumerate}
	If $N$ is $R$-oriented, then $\pi^*\cl R^*(N)\ra R^*(\psi_1\inv(0))$ is injective.
\end{prop}

\begin{proof} Set $K := \psi\inv(0)$ and $I := [0,1]$. By \eqref{proper}, $K$ is compact. Given any subset $W\sub \cY\times I$ we denote by $W_t$ its fibre over $t \in I$. Let $T$ be a stabilisation of $\psi$ along $K$ relative to $K_0\times\{0\}$ of rank $k$. Set $$S := (\psi + T)\inv(0)\sub \cY\times I\times\RE^k.$$ 
	Then $S$ is a smooth (non-compact) cobordism from $S_0$ to $S_1$ with $TS = \ker(d\psi + T)$. Assumption \eqref{proper} on $\psi$ implies that $S$ is $R$-oriented on a neighbourhood of $K$. By the compactness of $T(v,t,\cdot)$ for $(v,t)\in \cY\times I$ and \cite[Theorem A.1.5.i]{MS04},
	$$\dim(S) = n+k+1.$$ 
	Note that $K = \{(x,t,z)\in S : z= 0\}$ and $S_0 = K_0\times\RE^k$. Set $$\tilde{\pi} := \pi\times\ide_I\times\ide_{\RE^k}: S \ra N \times I \times \bR^k$$ and let $\tilde{\pi}_t$ be the restriction to $\cY\times\{t\}\times\RE^k$. This fits into a commutative diagram of pairs
	\begin{center}\begin{tikzcd}
		(S_0,S_0\sm K_0)\arrow[r,"\tilde{\pi}_0"] \arrow[d,hook,"i_0"] & (N\times \RE^k,N\times (\RE^k\sm0))\arrow[d,hook,"\iota_0"]\\ 
		(S,S\sm K)\arrow[r,"\tilde{\pi}"]  &  (N\times I\times \RE^k,N\times I\times (\RE^k\sm0))\\ 
		(S_1,S_1\sm K_1)\arrow[u,hook,"i_1"] \arrow[r,"\tilde{\pi}_1"] &  (N\times \RE^k,N\times(\RE^k\sm 0))\arrow[u,hook,"\iota_1"]\end{tikzcd} 
	\end{center}
	Consider the composition
	\begin{multline}\label{iems 1}R^*(N) \xrightarrow{\pi_1^*}
    R^*(K_1) \xrightarrow[\cong]{\mathrm{AD}}
    R_{n + k - *}(S_1, S_1 \setminus K_1) \\\xrightarrow{(\pi_1)_*}
    R_{n + k - *}(N \times \bR^k, N \times (\bR^k \setminus 0)) \xrightarrow[\cong]{\mathrm{AD}}
    R^*(N)\end{multline}
    where $\mathrm{AD}$ denotes the Atiyah duality isomorphism. Note that in the second map we use the $R$-orientability assumption in (1).\par
    By construction, this map is given by multiplication by $\mathrm{AD}((\pi_1)_* [S_1]_{K_1}) \in R^0(N)$,
	which is equal to $\mathrm{AD}((\pi_0)_* [S_0]_{K_0})$ by Lemma \ref{Fundamental Class of Boundary}.
	As $\pi_0$ is a diffeomorphism, this cohomology class is a unit. Hence, \eqref{iems 1} is an isomorphism. Because it factors through the pullback  $\pi_1^*: R^*(N) \rightarrow R^*(K_1)$, the latter must be injective.\end{proof}

\smallskip

We apply this to our situation. 
  Let $G\sub \bC$ be a simply-connected compact submanifold with smooth boundary, and suppose $\{L_z\}_{z \in \partial G}$ is a \emph{Hamiltonian family} of Lagrangians in $X$. That is, there exists a (relatively exact) Lagrangian $L \sub X$ and a smooth family $\{\phi^t_z\}_{z\in \partial G,t\in [0,1]}$ of Hamiltonian isotopies (which we can assume to be compactly supported) of $X$ such that $L_z = \phi_z^1(L)$ for all $z$. We can assume that $L = L_{z_0}$ for some $z_0 \in \partial G$.\\

Consider the following moduli space of pseudoholomorphic discs with moving Lagrangian boundary conditions:
     \begin{equation}\label{changing-boundary}\cP := \set{u \in C^\infty(G,X): \hpd_J u = 0, \;E(u) < \infty,\; \forall z \in\pd G: u(z)\in L_z}\end{equation}
    where $\hpd_J$ is the Cauchy-Riemann operator associated to $J$ and $E$ is the symplectic energy.
    Let $\pi: \cP \rightarrow L$ be evaluation at $z_0$.
    
\begin{thm}\label{Main Theorem 1}
    	The pullback $\pi^*: R^*(L) \rightarrow R^*(\cP)$ is injective.
\end{thm}

    \begin{rem}
        If the moduli space $\mathcal{P}$ were cut out transversely, this could be proved using a cobordism argument as in \cite{Por22}. On the other hand, following \cite[Remark 4.6]{Por22}, Theorem \ref{Main Theorem 1} can be used to give a slightly different proof of \cite[Corollary 1.9]{Por22}, without using any transversality results.
    \end{rem}

    \begin{rem}
        Hofer \cite{Ho88} proves Theorem \ref{Main Theorem 1} as well as Theorem \ref{Main Theorem 2}, \ref{Intersection Points and Cuplength} and Corollary \ref{corolary} in the case where $R^*$ is \v{C}ech cohomology with coefficients in $\bZ / 2$.
    \end{rem}

Using an extension of an associated family of Hamiltonians we may extend the family of Hamiltonian isotopies $\{\phi_z\}_{z\in \pd G}$ to a smooth family $\{\phi_z\}_{z\in G}$ of Hamiltonian isotopies, parametrised by $G$. \par 

Fix $k \geq 3$. Given $t \in [0,1]$, we define $\psi_t \cl W^{k,2}(G,X) \ra W^{k,2}(G,X)$ by setting 
$$\psi_t(u)(z) := \phi^t_z(u(z))$$
for $z \in G$. By assumption, $\psi_0$ is the identity. Let 
$$\cA:= \set{u\in W^{k,2}(G,X) : u(\pd G)\sub L}.$$

The smooth Banach bundle $\cE \ra \cA$ with fibre
$$\cE_u := W^{k-1,2}(G,u^*TX).$$
admits a smooth Fredholm section $\hpd_J: \cA \ra \cE$ given by 
$$\hpd_J u = \pd_s u + J(u)\pd_t u.$$

The canonical evaluation map defines a map of pairs $\eva \cl \cA\times (G,\pd G)\ra (X,L)$. By pulling back, this defines a bundle pair 
$$(F,F') := \eva^*(TX,TL) \ra \cA\times (G,\pd G).$$
Using a connection on $TX$, the linearisation of $\hpd_J$ defines a real Cauchy-Riemann operator on $(F,F')$ by \cite[Proposition 3.1.1]{MS04}. Then Assumption \ref{Assumption} states exactly that its index bundle 
is $R$-oriented. For $u \in \hpd_J\inv(0)$ we have, by the Riemann-Roch theorem, \cite[Theorem C.1.10]{MS04}, that
\begin{equation}\label{index 1}\indo(D_u\hpd_J ) = \dim(L)+\mu(F|_u,F'|_u), \end{equation}
where $\mu(F|_u,F'|_u)$ is the boundary Maslov index of the pullback of $(F,F')$ to $\{u\}\times G$. \par 

\begin{rem}\label{r}
    If $u \in \cA$ is pseudoholomorphic, then $\mu(F|_u,F'|_u) = 0$ as $u$ must be constant due to relative exactness. However, if $u$ instead satisfies that $\psi_t(u)$ is pseudoholomorphic for some $t > 0$, $u$ need not be constant and may lie in a non-trivial relative homotopy class of discs. In this case, $\mu(F|_u, F'|_u) $ might not vanish.
\end{rem}

By \cite[Theorem (3)]{Kui65} we can fix a smooth isometric trivialisation $\Psi \cl \cE\ra \cA\times \cH$, where $\cH$ is some separable Hilbert space.
Define $\cF\cl \cA\times[0,1] \ra \cH$ by
\begin{equation}\cF_t(u) := \pr_2\Psi(\hpd_J \psi_t(u))\end{equation}
letting $\pr_2$ denote the projection to the second factor. Note that $\cF_1^{-1}(0)$ is diffeomorphic via $\psi_1$ to the space $\cP$ of pseudoholomorphic maps from $G$ to $M$ which have finite energy and map $z\in \pd  G$ to $L_z$.

\begin{proof}[Proof of Theorem \ref{Main Theorem 1}]
    Let 
	$$\cW := \set{(u,t)\in \cA\times[0,1] : \mu(F|_{\psi_t(u)},F'|_{\psi_t(u)})= 0}.$$
	This is an open subset of $\cA\times [0,1]$. We restrict to the subset where the Maslov index is 0 in order to have control over the index of $\cF$, due to Remark \ref{r}. However, with a little more care the entire argument could also be applied without this restriction. Let $\pi: \cW \rightarrow L$ be the evaluation map at $z_0$.\par 
	
	By \cite[Proposition 6]{Ho88} there exists a neighbourhood $U \sub\cW$ of $\cF^{-1}(0)$ such that $\cF|_U: U \ra \cH$ is a Fredholm map of index $\dim(L)+1$ and such that $\cF|_U$ is proper with respect to a neighbourhood of $0 \in \cH$. We note that $U$ and $\cH$ are separable Hilbert manifolds. Since pseudoholomorphic discs with boundary on $L$ are constant, $\pi$ defines a diffeomorphism $\cF_0\inv(0)\ra L$. Moreover, $\cF_0$ is submersive by the proof of \cite[Lemma 5]{Ho88}, and $\cF$ is $R$-orientable by Assumption \ref{Assumption}. Thus the claim follows from Proposition \ref{Injectivity of Evaluation Map with Spectra}.
\end{proof}

\section{Approximating pseudoholomorphic strips}\label{Section 3}

The key idea in the proof of Theorem \ref{Main Theorem 2} is to study a one-parameter family of moduli spaces of pseudoholomorphic discs $\cP_\ell$ with moving boundary conditions. They approximate the moduli space of pseudoholomorphic strips $\cM_{L, L'}$ as the parameter $\ell$ tends to $\infty$. Together with \cite[Lemma 5.2]{AMS21}, this allows us to infer Proposition \ref{Main Theorem 2} from Theorem \ref{Main Theorem 1}.\par 

Throughout this section we fix a convex domain $G$ in $\bC$ with smooth boundary, such that both $(-\eta, \eta)$ and $(-\eta, \eta) + i$ are contained in $\partial G$ for some $\eta > 0$. For $\ell > 0$, define $Z_\ell := [-\ell, \ell] + [0, 1]i$, and let
$$G_\ell := {Z_\ell \cup (G + \ell) \cup (G - \ell)}$$
be a smoothing as shown below. \par
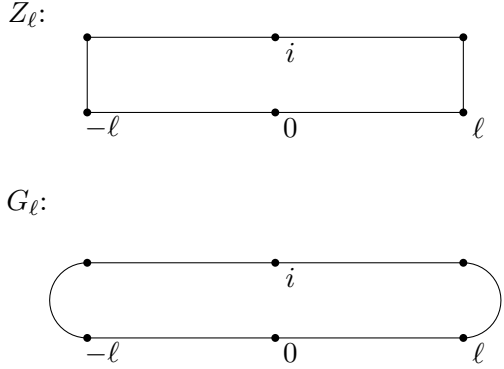
{\centering\begin{tikzpicture}[line cap=round,line join=round,>=triangle 45,x=1.0cm,y=1.0cm]
    \draw (0,0)-- (5,0);
    \draw (0,1)-- (5,1);
    \draw (5,1)-- (5,0);
    \draw (0,1)-- (0,0);            
    \draw (5,-3)-- (0,-3);
    \draw (0,-2)-- (5,-2);
    \node at (2.7,-0.2) {$ 0 $};
    \node at (2.7,0.8){$ i $};
    \node at (5.2,-0.2){$ \ell $};
    \node at (0.2,-0.2) {$ -\ell $};
    \node at (-0.8,1.3){$ Z_\ell\mathchar\numexpr"6000+`:\relax $};
    \node at (2.7,-3.2) {$ 0 $};
    \node at (2.7,-2.2) {$ i $};
    \node at (5.2,-3.2){$ \ell $};
    \node at (0.2,-3.2) {$ -\ell $};
    \node at (-0.8,-1.2) {$ G_\ell\mathchar\numexpr"6000+`:\relax  $};
    \draw [shift={(0,-2.5)}] plot[domain=1.57:4.71,variable=\t]({1*0.5*cos(\t r)+0*0.5*sin(\t r)},{0*0.5*cos(\t r)+1*0.5*sin(\t r)});
    \draw [shift={(5,-2.5)}] plot[domain=-1.57:1.57,variable=\t]({1*0.5*cos(\t r)+0*0.5*sin(\t r)},{0*0.5*cos(\t r)+1*0.5*sin(\t r)});
    \begin{scriptsize}
        \fill [color=black] (0,0) circle (1.5pt);
        \fill [color=black] (5,0) circle (1.5pt);
        \fill [color=black] (5,1) circle (1.5pt);
        \fill [color=black] (0,1) circle (1.5pt);
        \fill [color=black] (2.5,0) circle (1.5pt);
        \fill [color=black] (0,-2) circle (1.5pt);
        \fill [color=black] (5,-2) circle (1.5pt);
        \fill [color=black] (0,-3) circle (1.5pt);
        \fill [color=black] (5,-3) circle (1.5pt);
        \fill [color=black] (2.5,1) circle (1.5pt);
        \fill [color=black] (2.5,-2) circle (1.5pt);
        \fill [color=black] (2.5,-3) circle (1.5pt);
    \end{scriptsize}
\end{tikzpicture}\par }

\begin{de}
	For a domain $W$ in $\bC$ and a smooth map $u: W \rightarrow X$, we define the \emph{symplectic energy} to be
	$$E(u) := \frac12\int_W u^* \omega$$
	whenever this integral is defined.
\end{de}

When $u$ is pseudoholomorphic, the symplectic energy of $u$ is defined and non-negative, although not necessarily finite.\par 

We consider the following moduli spaces. Recall from the introduction that $L$ is a closed, relatively exact Lagrangian in $X$ and $L'$ is Hamiltonian isotopic to $L$. We denote by $Z := \bR+[0, 1]i $ the infinite strip.

\begin{de} We define
	$$\cD_{L, L'} := \set{u \in C^\infty(Z,X) : |E(u)|<  \infty,\; u(\bR)\sub L,\; u(\bR+i)\sub L'}.$$
	It contains $\cM_{L, L'} := \{u \in \cD_{L, L'}: \hpd_J u = 0\}$ as the subspace of pseudoholomorphic maps.\par
	Given $\ell > 0$ and $A \geq 0$, we set 
	$$\cF_{\ell, A} := \set{u \in C^\infty(Z_\ell,X): E(u) \leq A,\; \hpd_Ju = 0,\; u([-\ell,\ell])\sub L,\; u([-\ell,\ell]+i)\sub L'}.$$
	Given $\ell \geq 0$ and a Hamiltonian family $\{L_t\}_{t \in [0, 1]}$ of Lagrangians in $X$ with $L_0 = L$ and $L_1 = L'$, we define the moduli space
	$$\cP_\ell := \set{u\in C^\infty(G_\ell,X) : \hpd_J u = 0,\; u(s+i t)\in L_t \normalfont\text{ for } s +i  t \in \partial G_\ell}.$$
\end{de}

All of these spaces are endowed with the weak $C^\infty$ Whitney topology. By the Nash Embedding Theorem applied to the metric $g_J = \omega(\cdot,J\cdot)$, this topology is metrisable. Hofer showed in \cite[Theorems 1 and 2]{Ho88} that the moduli spaces $\cP_\ell$ and $\cM_{L, L'}$ are compact.\par
Evaluation at $0 \in \bC$ defines a continuous map, denoted by $\pi$, from each of these spaces to $L$.

\begin{rem} Pick some Hamiltonian isotopy $\{\psi^t\}_{t \in [0, 1]}$ such that $\psi^t(L) = L_t$ for all $t$. Setting
    $$L_{x + i y} := L_y\qquad \quad\text{and}\qquad \quad\psi^t_{x + i y} := \psi^{ty}$$
    for $x + i y$ in $\partial G$ shows that $\cP_\ell$ is the space of pseudoholomorphic maps $G_\ell \to M$ of finite energy which map $z \in \pd G_\ell$ to $L_z$, i.e, of the form \eqref{changing-boundary}. This allows us to apply Theorem \ref{Main Theorem 1} with $\cP = \cP_\ell$.
\end{rem}

We require the following uniform energy bound.

\begin{lem}[{\cite[Lemma 2]{Ho88}}]\label{Energy Uniformly Bounded}
	The symplectic energy is uniformly bounded on all $\cP_\ell$. More precisely, there exists a constant $C \geq 0$ such that for all $\ell > 0$ and all $u$ in $\cP_\ell$, we have $E(u) \leq C$.
\end{lem}

Fix a smooth cutoff function $\rho: \bR \rightarrow [0, 1]$ with
\begin{equation*}\label{}\rho(t) =
\begin{cases}
1 \quad& t \leq  \frac12\\
0 \quad& t \geq  \frac32\end{cases}\end{equation*}
and define for $\ell > 0$ the function $r_\ell: \cF_{\ell, A} \rightarrow \cD_{L, L'}$ by
$$r_\ell(u)(x + iy) := u(\rho(\ell\inv|x|)x+ i y).$$
By construction, $r_\ell(u)$ agrees with $u$ on $Z_{\frac\ell2}$.\\

We require the following result which one should consider as a variant of Gromov compactness.

\begin{prop}[{\cite[Proposition 3]{Ho88}}]\label{P3}
	For any neighbourhood $U$ of $\cM_{L, L'}$ in $\cD_{L, L'}$ and any $A \geq 0$, there exists $\ell_0 > 0$ such that $r_\ell(\cF_{\ell, A}) \subseteq U$ for all $\ell \geq \ell_0$.
\end{prop}

\smallskip

\begin{proof}[Proof of Proposition \ref{Main Theorem 2}]
	Let $C$ be the uniform energy bound from Lemma \ref{Energy Uniformly Bounded}. Any $u$ in any $\cP_\ell$ clearly satisfies $E(u|_{Z_\ell}) \leq C$. Picking $U$ an open neighbourhood of $\cM_{L, L'}$ in $\cD_{L, L'}$, and taking $\ell_0$ as in Proposition \ref{P3} with $A = C$, we obtain a commutative diagram
	\[\begin{tikzcd}
	\cP_{\ell_0} \arrow[r, "\cdot\vert_{Z_{\ell_0}}"]  \arrow[drr, "\pi"]&
	\cF_{\ell_0, C} \arrow[r, "r_{\ell_0}"] &
	U \arrow[d, "\pi"] \\
	& & L\\
	\end{tikzcd}\]
	By Theorem \ref{Main Theorem 1}, the pullback $\pi^*: R^*(L) \rightarrow R^*(U)$ must thus be injective. Using the isomorphism	$$R^*(\cM_{L, L'}) \cong \varinjlim  R^*(U)$$ 
	(taking a direct limit over open neighbourhoods of $\cM_{L, L'}$ in $\cD_{L, L'}$) given by \cite[Lemma 5.2]{AMS21} and the exactness of the direct limit functor, we may conclude.
\end{proof}

\begin{rem} \cite[Lemma 5.2]{AMS21} states that generalised cohomology theories as defined in \textsection \ref{2.1} satisfy the continuity axiom. This is one key ingredient for our generalisation of the results in \cite{Ho88}.
\end{rem}
\section{Lusternik-Schnirelmann theory}\label{Section 4}

Our goal in this section is the proof of Theorem \ref{Intersection Points and Cuplength}.  
Observe that there is a natural $\bR$-action on $\cM_{L, L'}$, by setting $t \cdot u := u(\cdot - t)$. The fixed points of this action are exactly the constant maps to points in $L \cap L'$. Hence there is a bijection between these fixed points and $L \cap L'$. 

\begin{lem}[{\cite[Lemma 4]{Ho88}}]
	There exists a continuous map $\sigma: \cM_{L, L'} \rightarrow \bR$ such that for any $u$ which is not a fixed point of the $\bR$-action, the function $t\mapsto \sigma(t \cdot u)$ is strictly decreasing.
\end{lem}
\begin{proof}[Sketch of the construction]  
	One should think of $\sigma$ as something akin to the Floer action functional. Indeed, if $X$ is Liouville and $L$ is an exact Lagrangian, we can take $\sigma$ to be the usual Floer action functional.\par 
	If not, for each path component $Q$ in $\cM_{L, L'}$, we fix a basepoint $u_0$ in $Q$, and define $\sigma(u_0) := 0$. Then for some other $u_1$ in $Q$, we pick a path $\{u_t\}_{t \in [0, 1]}$ from $u_0$ to $u_1$, and define 
	$$\sigma(u_1) = \int_{[0, 1]^2} v^* \omega$$
	where $v: [0, 1]^2 \rightarrow X$ is a smoothing (rel endpoints) of the map sending $(s, t)$ to $u_s(t i)$. This is well-defined due to relative exactness.
\end{proof}

Fix some basepoint $x_0\in L$. For any subset $S$ of $\cM_{L, L'}$ or $\cD_{L, L'}$, we consider the map of pairs 
$$\pi_S: (S, \emptyset) \rightarrow (L,x_0) : u \mapsto u(0),$$ 
as well as the pullback 
$$\pi_S^*: R^*(L,x_0) \rightarrow R^*(S).$$
\begin{de}
	To each subset $S$ of $\cM_{L, L'}$, we assign the non-negative integer 
	$$I(S) := \min\set{k \geq 1: \exists\; U_1, \ldots, U_k\sub \cM_{L, L'}\normalfont\text{ open } : S\sub U_1\cup\dots \cup U_k\normalfont\text{ and }  \pi^*_{U_i} = 0}.$$
\end{de}

Note $I$ has a uniform upper bound. Indeed, let $N$ be the  minimal number of contractible open subsets of $L$ required to cover $L$. Then  $I(S) \leq N$ for any $S\sub \cM_{L, L'}$.\par

\begin{lem}\label{Index Function}
	Fix subsets $S$ and $T$ of $\cM_{L, L'}$.
	\begin{enumerate}
	    \item If $S \subseteq T$, then $I(S) \leq I(T)$.
		\item There is some open neighbourhood $U$ of $S$ such that $I(S) = I(U)$.
		\item $I(S \cup T) \leq I(S) + I(T)$.
		\item If $\{\varphi_t\}_{t \in\bR}$ is a flow on $\cM_{L,L'}$, then $I(S) = I(\varphi_t(S))$ for all $t\in\bR$.
		\item\label{finite} $I(\{u_1,\dots,u_n\}) = 1$ for any $u_1,\dots,u_n\in \cM_{L,L'}.$
	\end{enumerate}
\end{lem}

Thus $I$ is an \emph{index function} in the sense of \cite[Definition 4.2]{Rud99}.

\begin{proof} If $S\leq T$, we take the minimum over a larger set, so the inequality is immediate. If $U_1, \ldots, U_{I(S)}$ are open subsets of $\cM_{L, L'}$ covering $S$ with $\pi^*_{U_i} = 0$ for all $i$, set 
		$$U = U_1 \cup \ldots \cup U_{I(S)}$$
		Then $I(U)\leq I(S)$, so equality holds by (1).
		The union of two suitable open covers for $S$, respectively $T$ defines a suitable open cover for $S\cup T$ which must have cardinality at least $I(S\cup T)$. This shows 3. As $\varphi_t$ is homotopic to the identity, it takes a suitable cover for $S$ to a suitable cover for $\varphi_t(S)$. Thus $I(\varphi_t(S)) \leq I(S)$, and equality follows from applying the same argument to $\varphi_{-t}(\varphi_t(S))$.
		To see the last claim, denote $\{p_1,\dots,p_k\} = \{u_1(0),\dotsc,u_n(0)\}$. For each $j \leq k$ choose a contractible open neighbourhood $V_j$ of $p_j$ in $L$, such that $\cc{V_i}\cap \cc{V_j}  =\emst$ for $i \neq j$. Then the preimage $U$ = $\pi^{-1}(V_1 \cup \dots \cup V_k)$ defines a suitable cover for $\{u_1,\dots,u_n\}$.
\end{proof}

\begin{lem}\label{Index Function and Cuplength}
	$I(\cM_{L, L'}) \geq c_R(L)$.
\end{lem}
\begin{proof}
    Fix an open cover $U_1, \ldots, U_k$ of $\cM_{L, L'}$ such that $\pi^*_{U_i} = 0$ for $i\leq k$ and let $\alpha_1,\dots,\alpha_k\in R^*(L,x_0)$ be arbitrary. By Lemma \ref{rel cup} the product $\pi_{\cM_{L, L'}}^* \alpha_1 \cdot \ldots \cdot \pi_{\cM_{L, L'}}^* \alpha_k$ vanishes in $R^*(\cM_{L, L'})$. By Theorem \ref{Main Theorem 1}, $\pi_{\cM_{L, L'}}^*$ is injective, so $\alpha_1 \cdot \ldots \cdot \alpha_k = 0$ in $R^*(L, x_0)$ and $c_R(L) \leq k$. Taking the infimum over all such open covers completes the proof.
\end{proof}

Given Lemma \ref{Index Function and Cuplength}, the proof of Theorem \ref{Intersection Points and Cuplength} reduces to showing that
\begin{equation}\label{Star} \# L \cap L' \geq I(\cM_{L, L'})\end{equation}

Since $I$ is an index function, this follows from \cite[Theorem 4.2]{Rud99}. For the sake of exposition, we give a proof here, using standard Lusternik-Schnirelman theory as in \cite[\textsection~V]{Ho88}.

\begin{de}
	For $1 \leq i \leq I(\cM_{L, L'})$, we define
	$$d_i := \inf_{I(S) \geq i} \sup \sigma(S)$$
	where the infimum is taken over subsets of $\cM_{L, L'}$.\par 
	For any $d\in \bR$, we denote
	$$Cr(d) := \set{u \in \cM_{L, L'} : \sigma(u) = d,\; \bR \cdot u = u }.$$
\end{de}

It follows that

$$\sum_d \# Cr(d) = \# L \cap L'.$$

\begin{lem}
	$$-\infty < d_1 \leq \ldots \leq d_{I(\cM_{L, L'})} < \infty.$$
\end{lem}
\begin{proof} First observe that
	$d_j \leq d_{j + 1}$ for all $j$ since we take the infimum over a smaller set. The compactness of $\cM_{L, L'}$ implies that 
	$-\infty < d_1$ and $d_{I(\cM_{L, L'})} < \infty$.
\end{proof}

\begin{lem}\label{ccc}
	For any neighbourhood $U$ of $Cr(d)$, there exists some $\varepsilon > 0$ such that
	$$u \in \sigma^{-1}((-\infty, d + \varepsilon])\setminus U\quad \rimp\quad 1\cdot u\in \sigma^{-1}((-\infty, d - \varepsilon]).$$
\end{lem}

\begin{proof}
	This follows from the compactness of $\sigma^{-1}((-\infty, d]) \setminus U$ along with the continuity of the $\bR$-action.
\end{proof}

\begin{lem}\label{aaa}
	$Cr(d_j)$ is non-empty for all $j$.
\end{lem}

\begin{proof}
	Suppose $Cr(d_j)$ is empty. Applying Lemma \ref{ccc} to $U = \emst$ we obtain some $\varepsilon > 0$ such that
	$$1 \cdot\sigma\inv((-\infty, d_j + \varepsilon]) \subseteq \sigma^{-1}((-\infty, d_j - \varepsilon]).$$
	By definition of $d_j$, there exists $S \subseteq \cM_{L, L'}$ such that $I(S) \geq j$ and 
	$$ d_j \leq \sup \sigma(S) \leq d_j + \varepsilon.$$
	But then $I(1 \cdot S) \geq j$ and $\sup \sigma(1 \cdot S) < d_j$, which is a contradiction.
\end{proof}

\begin{lem}\label{bbb}
	If $d_j = d_{j + 1}$ for any $j$, then $Cr(d_j)$ is infinite.
\end{lem}

\begin{proof}
	If $Cr(d_j)$ is finite, then $I(Cr(d_j)) = 1$ by Lemma \ref{Index Function}.\eqref{finite}. So it suffices to show that $I(Cr(d_j)) \geq 2$.
	Suppose by contradiction $I(Cr(d_j)) \leq 1$. Since $Cr(d_j)$ is non-empty, we must have equality. Then there is some open neighbourhood $U$ of $Cr(d_j)$ such that $\pi^*_{U} = 0$. Given this $U$, fix $\varepsilon > 0$ as in the statement of Lemma \ref{ccc}.\par 
	Choose $S \subseteq \cM_{L, L'}$ such that $I(S) \geq j + 1$ and 
	$$d_j \leq \sup \sigma(S) \leq d_j + \varepsilon.$$
	Then $I(1 \cdot (S \setminus U)) \geq j$ but $\sigma(1 \cdot (S \setminus U)) \leq d_j - \varepsilon$, a contradiction.
\end{proof}

The inequality in \eqref{Star}, and hence Theorem \ref{Intersection Points and Cuplength}, follows from Lemmas \ref{aaa} and \ref{bbb}.
\section{Two examples}\label{Examples}

Let
$$\Sp(n) := \normalfont\text{Sp}(2n;\bC)\cap U(2n)$$ be the \emph{compact symplectic group}. It is a compact simply-connected Lie group of dimension $n(2n+1)$. The zero section defines a Lagrangian embedding $\Sp(n)\hkra T^*\Sp(n)$, where we endow $T^*\Sp(n)$ with the canonical symplectic structure. As this embedding is a homotopy equivalence, $\Sp(n)$ is relatively exact. 
We will consider $\Sp(2)$ and $\Sp(3)$ since their cuplength with respect to a certain generalised cohomology theory was computed in \cite{IM04} (see also \cite{Ki07}) and is strictly greater than their cuplength with respect to integral cohomology.

\begin{prop}[\cite{IMN03, IM04}] \label{computations}
    The mod-2 and integral cuplengths of $\Sp(2)$ are
    $$c_{\bZ/2}(\Sp(2)) = c_\bZ(\Sp(2)) = 3,$$
    while 
    $$c_{h^*}(\Sp(2)) = 4$$
    where $h^*$ is the cohomology theory associated to the truncated sphere spectrum $\bS[0, 2]$. In particular, \cite{IM04} shows that its cuplength in real $K$-theory is
    $$c_{KO}(\Sp(2)) = 4.$$
    Similarly the cuplengths of $\Sp(3)$ in the same cohomology theories are given by
    $$c_{\bZ/2}(\Sp(3)) = c_\bZ(\Sp(3)) = 4$$
    and
    $$c_{h^*}(\Sp(3)) = c_{KO}(\Sp(3)) = 5.$$
\end{prop}

 We remind the reader that we use Hofer's definition of the cuplenght, which differs by one compared to that of \cite{IM04}.\par 

Since $\Sp(2)$ and $\Sp(3)$ are Lie groups, they are parallelisable. By Proposition \ref{Orientability Holds Often} we can therefore apply Theorem \ref{Intersection Points and Cuplength} with real $K$-theory to either one as the zero section lying inside its cotangent bundle. This gives a (strictly) stronger bound on the Arnol'd number than Hofer's cup length estimate, though this estimate was already known due to work of Laudenbach and Sikorav \cite{LS85}, using finite-dimensional approximations.

\begin{cor}\label{corollary}
    The minimum number of intersection points between a relatively exact Lagrangian embedding of $\Sp(2)$ (satisfying Proposition \ref{Orientability Holds Often}.\ref{sphere}) and its image under any Hamiltonian diffeomorphism is at least $4$. The same is true for $\Sp(3)$ with $5$ instead of $4$.
\end{cor}

\begin{prop}
    The critical number of $\Sp(2)$ is 4.
\end{prop}
\begin{proof}
    The critical number of $\Sp(2)$ is bounded below by 4 by Proposition \ref{computations}. On the other hand, \cite[Theorem 6.1]{Sm62} combined with the computation of its integral homology in Lemma \ref{Z_2 cuplength of Sp(2)} implies that the Morse number of $\Sp(2)$ is $4$, which is an upper bound for the critical number.
\end{proof}

The bound for $\Sp(2)$ (and a stronger bound for $\Sp(3)$) were shown by \cite{LS85} for the respective zero section in the cotangent bundle. However, their methods are specifically geared towards cotangent bundles while our bounds persists under Weinstein handle attachments. As an example, one can plumb a copy of $T^*\Sp(2)$ containing $\Sp(2)$ as the zero section, with the cotangent bundle of any other $2$-connected manifold of the same dimension to obtain a new Weinstein manifold $X$. Since $\Sp(2)$ is $2$-connected, the resulting manifold admits the the same $2$-skeleton (up to homotopy) as $T^*\Sp(2)$, which is trivial. Therefore Proposition \ref{Orientability Holds Often}\ref{sphere} still holds for the embedding $\Sp(2)\hkra X$. This gives a stronger bound than Hofer's estimate, in a case for which the estimate of \cite{LS85} does not apply.\\  

We recap the computation of the integral cohomology rings of $\Sp(2)$ and $\Sp(3)$ here.

\begin{lem}[\cite{IM04}]\label{Z_2 cuplength of Sp(2)} The integral cuplengths of $\Sp(2)$ and $\Sp(3)$ are given by 
$$c_{\bZ}(\Sp(2)) = 3\quad\text{and}\quad c_{\bZ}(\Sp(3)) = 4.$$ Furthermore, $H^*(\Sp(2))$ and $H^*(\Sp(3))$ are free of rank 4 and 8 respectively, over both $\bZ$ and $\bZ/2$.
\end{lem}

\begin{proof} Given $n$ we can identify $\bH^n$ with $\bR^{4n}$ to see the existence of a principal $\Sp(n-1)$-bundle $\Sp(n)\to S^{4n-1}$ induced by the canonical action on the unit quaternions. Thus we can apply the Leray-Serre spectral sequence to compute the cohomology of $\Sp(2)$ with coefficients in $A = \bZ$ or $A = \bZ/2$. The $E_2$-page is given by 
$$E^{p,q}_2 = H^p(S^7,H^q(\Sp(1);A))$$
which vanishes for $p \notin\{0,7\}$ and $q \notin\{0, 3\}$. Hence the spectral sequence collapses for degree reasons at the second page. As $\Sp(1) \cong S^3$, we obtain
\begin{equation}\label{iso-cohom}H^*(\Sp(2); A) \cong H^*(S^7; A) \otimes_A H^*(S^3; A).\end{equation}
By the multiplicativity of the spectral sequence, \eqref{iso-cohom} is an isomorphism of graded rings. Therefore,
$$H^n(\Sp(2);A) = \begin{cases}
A\quad & n \in \{0,3,7,10\},\\
0 \quad &\normalfont\text{otherwise}
\end{cases}$$
and we can deduce the values of $c_{\bZ/2}(\Sp(2))$ and $c_\bZ(\Sp(2))$.\par 
The same argument gives an isomorphism of graded rings 
$$H^*(\Sp(3);A) \cong H^*(S^{11};A) \otimes_A H^*(\Sp(2);A).$$
Therefore,
$$H^n(\Sp(3);A) = \begin{cases}
A\quad & n \in \{0,3,7,10, 11, 14, 18, 21\},\\
0 \quad &\normalfont\text{otherwise}
\end{cases}$$
from which we deduce the corresponding statements for $\Sp(3)$.
\end{proof}

\bibliographystyle{alpha}
\bibliography{bibglch}

\begin{thebibliography}{FOOO09}

\bibitem[AB21]{AB21}
Mohammed Abouzaid and Andrew~J. Blumberg.
\newblock Arnold conjecture and {M}orava {K}-theory, 2021.

\bibitem[Ada95]{Ad95}
Frank Adams.
\newblock {\em Stable homotopy and generalised homology}.
\newblock Chicago Lectures in Mathematics. University of Chicago Press, Chicago, IL, 1995.
\newblock Reprint of the 1974 original.

\bibitem[AH16]{AH16}
Peter Albers and Doris Hein.
\newblock Cuplength estimates in morse cohomology.
\newblock {\em Journal of Topology and Analysis}, 08(02):243--272, 2016.

\bibitem[AMS21]{AMS21}
Mohammed Abouzaid, Mark McLean, and Ivan Smith.
\newblock Complex cobordism, {H}amiltonian loops and global {K}uranishi charts, 2021.
\newblock arXiv:2103.01507.

\bibitem[AMS23]{AMS23}
Mohammed Abouzaid, Mark McLean, and Ivan Smith.
\newblock Gromov-{W}itten invariants in complex-oriented generalised cohomology theories, 2023.
\newblock arXiv:2307.01883.

\bibitem[BHS21]{BHS21}
Lev Buhovsky, Vincent Humiliere, and Sobhan Seyfaddini.
\newblock The action spectrum and c 0 symplectic topology.
\newblock {\em Mathematische Annalen}, 380(1):293--316, 2021.

\bibitem[BX22]{BX22}
Shaoyun Bai and Guangbo Xu.
\newblock Arnold conjecture over integers, 2022.

\bibitem[CJS95]{CoJoSe95}
Ralph~L. Cohen, John D.~S. Jones, and Graeme~B. Segal.
\newblock Floer's infinite-dimensional {M}orse theory and homotopy theory.
\newblock In {\em The {F}loer memorial volume}, volume 133 of {\em Progr. Math.}, pages 297--325. Birkh\"{a}user, Basel, 1995.

\bibitem[CLOT03]{CLOT03}
Octav Cornea, Gregory Lupton, John Oprea, and Daniel Tanr\'{e}.
\newblock {\em Lusternik-{S}chnirelmann category}, volume 103 of {\em Mathematical Surveys and Monographs}.
\newblock American Mathematical Society, Providence, RI, 2003.

\bibitem[Coh09]{CJS09}
Ralph~L. Cohen.
\newblock Floer homotopy theory, realizing chain complexes by module spectra, and manifolds with corners.
\newblock In {\em Algebraic topology}, volume~4 of {\em Abel Symp.}, pages 39--59. Springer, Berlin, 2009.

\bibitem[FL23]{FL23}
Oliver Fabert and Niek Lamoree.
\newblock Cuplength estimates for periodic solutions of hamiltonian particle-field systems.
\newblock {\em Journal of Fixed Point Theory and Applications}, 25(2):47, 2023.

\bibitem[Flo88]{Fl88}
Andreas Floer.
\newblock Morse theory for {L}agrangian intersections.
\newblock {\em J. Differential Geom.}, 28(3):513--547, 1988.

\bibitem[Flo89]{Fl89}
Andreas Floer.
\newblock Cuplength estimates on {L}agrangian intersections.
\newblock {\em Comm. Pure Appl. Math.}, 42(4):335--356, 1989.

\bibitem[FO99]{FO99}
Kenji Fukaya and Kaoru Ono.
\newblock Arnold conjecture and {G}romov-{W}itten invariant.
\newblock {\em Topology}, 38(5):933--1048, 1999.

\bibitem[FOOO09]{FOOO11}
Kenji Fukaya, Yong-Geun Oh, Hiroshi Ohta, and Kaoru Ono.
\newblock {\em Lagrangian intersection {F}loer theory: anomaly and obstruction. {P}art {I}}, volume~46 of {\em AMS/IP Studies in Advanced Mathematics}.
\newblock American Mathematical Society, Providence, RI; International Press, Somerville, MA, 2009.

\bibitem[GG20]{GiGu20}
Viktor~L. Ginzburg and Ba\c{s}ak~Z. G\"{u}rel.
\newblock Lusternik-{S}chnirelmann theory and closed {R}eeb orbits.
\newblock {\em Math. Z.}, 295(1-2):515--582, 2020.

\bibitem[Gon21a]{Go21b}
Wenmin Gong.
\newblock Lagrangian {L}justernik-{S}chnirelman theory and {L}agrangian intersections, 2021.

\bibitem[Gon21b]{Go21a}
Wenmin Gong.
\newblock A short proof of cuplength estimates on {L}agrangian intersections, 2021.

\bibitem[Hof85]{Ho85}
Helmut Hofer.
\newblock Lagrangian embeddings and critical point theory.
\newblock {\em Ann. Inst. H. Poincar\'{e} Anal. Non Lin\'{e}aire}, 2(6):407--462, 1985.

\bibitem[Hof88]{Ho88}
Helmut Hofer.
\newblock Lusternik-{S}chnirelman-theory for {L}agrangian intersections.
\newblock {\em Ann. Inst. H. Poincar\'{e} Anal. Non Lin\'{e}aire}, 5(5):465--499, 1988.

\bibitem[IM04]{IM04}
Norio Iwase and Mamoru Mimura.
\newblock {\em L-{S} Categories of Simply-connected Compact Simple {L}ie Groups of Low Rank}, pages 199--212.
\newblock Birkh{\"a}user Basel, Basel, 2004.

\bibitem[IMN03]{IMN03}
Norio {Iwase}, Mamoru {Mimura}, and Tetsu {Nishimoto}.
\newblock {Lusternik-{S}chnirelmann categories of non-simply connected compact simple Lie groups}.
\newblock {\em arXiv Mathematics e-prints}, page math/0303085, March 2003.

\bibitem[Jam78]{J78}
I.M. James.
\newblock On category, in the sense of lusternik-schnirelmann.
\newblock {\em Topology}, 17(4):331--348, 1978.

\bibitem[Kis07]{Ki07}
Daisuke Kishimoto.
\newblock L-{S} category of quaternionic {S}tiefel manifolds.
\newblock {\em Topology and its Applications}, 154(7):1465--1469, 2007.
\newblock Special Issue: The Third Joint Meeting Japan-Mexico in Topology and its Applications.

\bibitem[Kli78]{Kl78}
Wilhelm Klingenberg.
\newblock {\em Lectures on closed geodesics}.
\newblock Grundlehren der Mathematischen Wissenschaften, Vol. 230. Springer-Verlag, Berlin-New York, 1978.

\bibitem[Kui65]{Kui65}
Nicolaas~H. Kuiper.
\newblock The homotopy type of the unitary group of {H}ilbert space.
\newblock {\em Topology}, 3:19--30, 1965.

\bibitem[LO96]{VaOn96}
H\^{o}ng~V\^{a}n L\^{e} and Kaoru Ono.
\newblock Cup-length estimates for symplectic fixed points.
\newblock In {\em Contact and symplectic geometry ({C}ambridge, 1994)}, volume~8 of {\em Publ. Newton Inst.}, pages 268--295. Cambridge Univ. Press, Cambridge, 1996.

\bibitem[LS85]{LS85}
Fran\c{c}ois Laudenbach and Jean-Claude Sikorav.
\newblock Persistance d'intersection avec la section nulle au cours d'une isotopie hamiltonienne dans un fibr\'{e} cotangent.
\newblock {\em Invent. Math.}, 82(2):349--357, 1985.

\bibitem[MS04]{MS04}
Dusa McDuff and Dietmar Salamon.
\newblock {\em {$J$}-holomorphic curves and symplectic topology}, volume~52 of {\em American Mathematical Society Colloquium Publications}.
\newblock American Mathematical Society, Providence, RI, 2004.

\bibitem[MS17]{MS17}
Dusa McDuff and Dietmar Salamon.
\newblock {\em Introduction to symplectic topology}, volume~27 of {\em Oxford Graduate Texts in Mathematics}.
\newblock Ofxford University Press, third edition, 2017.

\bibitem[Oh93]{Oh93}
Yong-Geun Oh.
\newblock Floer cohomology of {L}agrangian intersections and pseudo-holomorphic disks. {I}.
\newblock {\em Comm. Pure Appl. Math.}, 46(7):949--993, 1993.

\bibitem[Opr14]{Op14}
John Oprea.
\newblock Applications of lusternik-schnirelmann category and its generalizations.
\newblock {\em Journal of Geometry and Symmetry in Physics}, 36:59--97, 01 2014.

\bibitem[OR99]{OR99}
John Oprea and Yuli~B. Rudyak.
\newblock On the {L}usternik-{S}chnirelmann category of symplectic manifolds and the {A}rnold conjecture.
\newblock {\em Math. Z.}, 230(4):673--678, 1999.

\bibitem[Par16]{P16}
John Pardon.
\newblock An algebraic approach to virtual fundamental cycles on moduli spaces of pseudo-holomorphic curves.
\newblock {\em Geom. Topol.}, 20(2):779--1034, 2016.

\bibitem[Por22]{Por22}
Noah Porcelli.
\newblock Families of relatively exact {L}agrangians, free loop spaces and generalised homology, 2022.
\newblock arXiv:2202.09677.

\bibitem[Rez22]{Re22}
Semon Rezchikov.
\newblock Integral {A}rnol'd {C}onjecture, 2022.

\bibitem[Rud98]{Rud98}
Yuli~B. Rudyak.
\newblock {\em On {T}hom spectra, orientability, and cobordism}.
\newblock Springer Monographs in Mathematics. Springer-Verlag, Berlin, 1998.
\newblock With a foreword by Haynes Miller.

\bibitem[Rud99]{Rud99}
Yuli~B. Rudyak.
\newblock On analytical applications of stable homotopy (the {A}rnold conjecture, critical points).
\newblock {\em Math. Z.}, 230(4):659--672, 1999.

\bibitem[Sch98]{S98}
Matthias Schwarz.
\newblock A quantum cup-length estimate for symplectic fixed points.
\newblock {\em Inventiones mathematicae}, 133(2):353--397, 1998.

\bibitem[Sma62]{Sm62}
S.~Smale.
\newblock On the structure of manifolds.
\newblock {\em Amer. J. Math.}, 84:387--399, 1962.

\bibitem[tD08]{tD08}
Tammo tom Dieck.
\newblock {\em Algebraic topology}.
\newblock EMS Textbooks in Mathematics. European Mathematical Society (EMS), Z\"{u}rich, 2008.

\end{thebibliography}

\Addresses

\end{document}